\documentclass[12pt,twoside]{amsart}
\usepackage{amsmath}
\usepackage{amssymb}
\usepackage{amscd}
\usepackage{xypic}
\xyoption{all}
\title[Curvature of Metrics on Semple Jet bundles]{Curvature of Metrics on Semple Jet bundles}

\author{Mohammad Reza Rahmati}
\address{Institut Fourier, Universitet Gronoble Alpes, Grenoble, France
\hfill\break 
\hfill\break \\
\hfill\break }
\email{mrahmati@cimat.mx}

\newcommand{\comments}[1]{}

\newtheorem{theorem}{Theorem}[section]
\newtheorem{proposition}[theorem]{Proposition}

\newtheorem{remark}[theorem]{Remark}



\begin{document}

\begin{abstract}
We discuss the Morse estimates for the curvature of several metrics on Semple weighted projective bundle over a projective variety. Following Demailly works on holomorphic Morse inequalities we show an analogue of his results along the Green-Griffiths conjecture for invariant jets.
\end{abstract}

\maketitle


\section{Introduction}

We give an equivariant (jet coordinate free) version of Morse cohomology estimates of J. P. Demailly \cite{D1}, \cite{D2} for invariant k-jet metrics on Demailly-Semple bundles. The result will have the same effect toward the Green-Griffith conjecture on the entire curve locus in projective manifolds. M. Green and P. Griffiths \cite{GG} give several equivalent construction of the jet bundle together with a method of calculation of their Chern classes. They also give an application toward the Kobayashi hyperbolicity conjecture. In \cite{D1} Demailly proves the existence of global (weighted) homogeneous sections of $k$-jets by holomorphic Morse inequalities. By standard facts any entire curve $f:\mathbb{C} \to X$ is satisfied by such sections. It is preferable to give the same result using coordinate free jet bundles of Semple and Demailly. 

In this article we generalize the main result of Demailly \cite{D2} for the bundles $E_{k,m}^{GG}(V^*)$ of jet differentials of order $k$ and weighted degree $m$ to the bundles $E_{k,m}(V^*)$ of the invariant jet differentials of order $k$ and weighted degree $m$. Namely, Theorem 0.5 from \cite{D2} and Theorem 9.3 from \cite{D1} provide a lower bound $\frac{c^k}{k}m^{n+kr-1}$ on the number of the linearly independent holomorphic global sections of $E_{k,m}^{GG} V^* \bigotimes \mathcal{O}(-m \delta A)$ for some ample divisor $A$. The group $G_k$ of local reparametrizations of $(\mathbb{C},0)$ acts on the $k$-jets by orbits of dimension $k$, so that there is an automatic lower bound $\frac{c^k}{k} m^{n+kr-1}$ on the number of the linearly independent holomorphic global sections of $E_{k,m}V^* \bigotimes \mathcal{O}(-m \delta A)$.

P. Griffiths and M. Green \cite{GG} provided several equivalent methods for characterizing the bundle of jets over a projective variety $X$. Later, the techniques developed by Semple, Demailly, McQuillan, Siu and Green-Griffiths led to new conjectures in Kahler geometry. Demailly \cite{D1}, \cite{D2} shows the existence of global sections of the Green-Griffiths bundles with sufficiently high degree. 
One can define the Green-Griffiths bundle $X_k$ as 

\begin{equation}
X_k: = (J_kV \smallsetminus {0}) / \mathbb{C}^* 
\end{equation}

\noindent
where $J_k$ is the bundle of germs of $k$-jets of Taylor expansion for $f$. The bundles $\pi_k:X_k \to X$ provide a tool to study entire holomorphic curves $ f: \mathbb{C} \to X $, since such a curve has a lift to $f_{[k]}:\mathbb{C} \to  X_k $ for every $ k $. Using the notation in \cite{D1}, we write $E_{k, m}^{GG}V^* =(\pi_{k})_* \mathcal{O}_{X_k}(m)$. 

Alternatively; the Green-Griffith bundle associated to a pair $(X,V)$ where $X$ is a complex projective manifold (may be singular) and $V$ is a holomorphic subbundle of $T_X$ the tangent bundle of $X$, with $rank(V)= r$, is a prolongation sequence of projective bundles 

\begin{equation}
\mathbb{P}^{r-1} \to X_k=P({V_{k-1}}) \stackrel{\pi_k}{\longrightarrow}  X_{k-1}, \qquad k \geq 1
\end{equation} 

\noindent
obtained inductively making $X_k$ a weighted projective bundle over $X$ (see \cite{D1} and \cite{D2} for definitions). The sequence provides a tool to study the locus of nonconstant holomorphic maps $f:\mathbb{C} \to X$ such that $f'(t) \in V$. It is a conjecture due to Green-Griffiths that the total image of all these curves is included in a proper subvariety of $X$; provided $X$ is of \textbf{general type}. By general type we mean $K_{V}$ the canonical bundle of $V$ is big, cf. \cite{D1}. An alternative way to define $X_k$, is to identify its sections by the weighted homogeneous polynomials in the jet variables $\xi_1,..., \xi_k$ with weights $(1,2,...,k)$ respectively, and coefficients to be germs of analytic functions on $X$. We denote by $E_{k,m}^{GG}$ the sections of total weight $m$.

\begin{theorem}\cite{D1}
Lets $(X,V)$ be a directed projective variety such that $K_V$ is big, and let $A$ be an ample divisor. Then for $k>>1$ and $\delta \in \mathbb{Q}_+$ small enough, and $\delta \leq c (\log k)/k$, the number of sections $h^0(X,E_{k,m}^{GG} \otimes \mathcal{O}(-m \delta A))$ has maximal growth, i.e is larger than $c_km^{n+kr-1}$ for some $m \geq m_k$, where $c,c_k >0, n=\dim (X), r= \text{rank}(V)$. In particular entire curves $f:\mathbb{C} \to X$ satisfy many algebraic differential equations.
\end{theorem}

\noindent
The proof of the above theorem is mainly an estimate on the curvature of a suitable metric  ($k$-jet metric) namely $h_k$ on the Green-Griffiths jet bundles. The singularity locus for the metric $h_k$ which we denote by $\Sigma_{h_k}$ satisfies the inductive relation

\begin{equation}
\Sigma_{h_k} \subset \pi_k^{-1}(\Sigma_{h_{k-1}}) \cup D_k
\end{equation}

\noindent
where $D_k=P(T_{X_{k-1}/X_{k-2}}) \subset X_k$. The divisors $D_k$ are the singularity locus of the projective jet bundle $X_k$ and their relation with the singularity of the k-jet metric is  $\mathcal{O}_{X_k}(1)=\pi_k^*\mathcal{O}_{X_{k-1}}(1) \otimes \mathcal{O}(D_k)$. 

\begin{theorem} (\cite{D1})
Let $(X,V)$ be a compact directed manifold. If $(X,V)$ has a $k$-jet metric $h_k$ with negative jet curvature, then every entire curve $f:\mathbb{C} \to X$ tangent to $V$ satisfies $f_k(\mathbb{C}) \subset \Sigma_{h_k}$, where $\Sigma_{h_k}$ is the singularity locus of $h_k$.
\end{theorem}

\noindent
The theorem provides a method to read the entire curve locus from the singularities of suitable metrics on the jet bundle, see \cite{D1} and \cite{D2} for details.

Any entire curve $ f: \mathbb{C} \to X $ satisfies $Q(f_{[k]})=0$, where $f_{[k]}$ a lift of $f$ and $Q$ a global section of the jet bundle (see \cite{D1}). As the section $Q$ involves both the manifold and jet coordinates, it is desirable to show that there are enough dual global sections along the jet fibers to pair with the given section to give equations on the variety $X$. We prove a Serre duality for asymptotic sections of jet bundles. An application is also given for partial application to the Green-Griffiths conjecture.

\vspace{0.1cm}

\section{Invariant Jets vs Invariant metrics}

\vspace{0.3cm}

J. P. Demaily develops the ideas in \cite{GG} and considers the jets of differentials that are also invariant under change of coordinate on $\mathbb{C}$. The bundle $J_k \to X$ of $k$-jets of germs of parametrized curves in $X$ has its fiber at $x \in X$ the set of equivalence classes of germs of holomorphic maps $f:(\mathbb{C},0) \to (X,x)$ with equivalence relation $f^{(j)}(0)=g^{(j)}(0), \ 0 \leq j \leq k$. By choosing local holomorphic coordinates around $x$, the elements of the fiber $J_{k,x}$ can be represented by the Taylor expansion 

\begin{equation}
f(t)=tf'(0)+\frac{t^2}{2!}f''(0)+...+\frac{t^k}{k!}f^{(k)}(0)+O(t^{k+1})
\end{equation}

\noindent
Setting $f=(f_1,...,f_n)$ on open neighborhoods of $0 \in \mathbb{C}$, the fiber is 

\begin{equation}
J_{k,x}=\{(f'(0),...,f^{(k)}(0))\} = \mathbb{C}^{nk}
\end{equation}

\noindent
Lets $G_k$ be the group of local reparametrizations of $(\mathbb{C},0)$

\begin{equation}
t \longmapsto \phi(t)=a_1t+a_2t^2+...+a_kt^k+..., \qquad a_1 \in \mathbb{C}^* 
\end{equation}

\noindent
Its action on the $k$-jet vectors (2.2) is given by the following matrix multiplication

\begin{equation}
[f'(0), f''(0)/2!, ..., f^{(k)}(0)/k!].
\left[ \begin{array}{ccccc}
a_1 & a_2 & a_3 & ... & a_k \\
0 & a_1^2 & 2a_1a_2 & ... & a_1a_{k-1}+...a_{k-1}a_1\\
0 & 0 & a_1^3 & ... & 3a_1^2a_{k-2}+...\\
. & . & . & ... & .\\
0 & 0 & 0 & ... & a_1^k 
\end{array} \right ]
\end{equation}

\noindent
The group $G_k$ decomposes as $\mathbb{C}^* \times U_k$, where $U_k$ is the unipotent radical of $G_k$. Let $E_{k,m}$ be the Demaily-Semple bundle whose fiber at $x$ consists of $U_k$-invariant polynomials on the fiber coordinates of $J_k$ at $x$ of weighted degree $m$. Set $E_k=\bigoplus_m E_{k,m}$, the Demaily-Semple bundle of graded algebras of invariant jets. Then one needs to work out the calculations in \cite{D1}, \cite{D2} with an invariant metric. Toward this we examine the following metric

\begin{equation}
\left (\sum_{s=1}^k \epsilon_s \left (\ \sum_{\alpha} \mid P_{\alpha}(\xi) \mid ^2 \right )^{\frac{p}{w(P_{\alpha})}}\right )^{1/p}
\end{equation}

\noindent
where the $P_{\alpha}$ are a set of invariant polynomials in jet coordinates. The effect of this is then, the Demailly-Semple locus of the lifts of entire curves should be contained in 

\begin{equation}
\Sigma_{h_k} \subset\  \{ P_{\alpha}=0 , \ \  \forall \alpha\}
\end{equation} 

\noindent
where the jet coordinates on the fibres $J_{k,x}$ are $\xi_1,...,\xi_k \in \mathbb{C}^n$ with $\xi_i=f^{(i)}(0)$ for an entire holomorphic curve $f : \mathbb{C} \longrightarrow X$ tangent to $V$, Besides,
$\epsilon_1 \gg \epsilon_2 \gg . . . \gg \epsilon_k > 0$ are sufficiently small and $w(P_{\alpha})$ is the weight of $P_{\alpha}$.

For instance a choice of $P_{\alpha}$'s could be the Wronskians. However we slightly try to do some better choice. First lets make some correspondence between invariant jets with non-invariant ones. Lets consider a change of coordinates 

\begin{equation}
(f_1,...,f_r) \longmapsto (f_1 \circ f_1^{-1},...,f_r \circ f_1^{-1})=(t,g_2,...,g_r)=\eta
\end{equation} 

\noindent
locally defined in a neighborhood of a point, where $n-r$ coordinates $f_i$ of $f(t) \in V_{f(t)}$ are completely determined by the remaining $r$ coordinates. This makes the first coordinate to be the identity. If we differentiate in the new coordinates successively, then all the resulting fractions are invariant of degree $0$,

\begin{equation}
g_2'=\frac{f_2'}{f_1'}\circ f_1^{-1}, \qquad g_2''=\frac{f_1'f_2''-f_2'f_1''}{f_1'^3},...
\end{equation} 

\noindent
We take the $P_{\alpha}$'s to be all the polynomials that appear in the numerators of the components when we successively differentiate (2.7) with respect to $t$. An invariant metric in the first coordinates corresponds to a usual metric in the second one subject to the condition that we need to make the average under the unitary change of coordinates in $V$. It corresponds to change of coordinates on the manifold $X$ as 

\begin{equation}
(\psi \circ f)^{(k)}(0)= \psi'(0). f^{(k)}+\text{higher order terms according to epsilons }
\end{equation} 

\noindent
with $\psi$ to be unitary. That is the effect of the change of variables in $X$ has only effect as the first derivative by composition with a linear map, up to the scaling epsilon factors (cf. \cite{D1}).  

\begin{equation}
\mid (z;\xi)\mid \sim \left (\sum_s  \epsilon_s \parallel \eta_s . (\eta_{11})^{2s-1}\parallel_h^{p/(2s-1)}\right )^{1/p} =\left (\sum_s  \epsilon_s \parallel \eta_s \parallel_h^{p/(2s-1)} \right )^{1/p}\mid \eta_{11} \mid
\end{equation}

\noindent
where $\eta_s$ are the jet coordinates $\eta_s = g^{(s)}(0), 1 \leq s \leq k$, induced by $g = (t, g_2(t), . . . , g_r(t))$. The weight of $\eta_s$ can be seen by differentiating (9) to be equal $(2s-1)$ inductively. Therefore the above metric becomes similar to the metric used by Demailly in the new coordinates produced by $g$'s. We need to modify the metric in (2.10) slightly to be invariant under hermitian transformations of the vector bundle $V$. In fact the role of $\eta_{11}$ can be done by any other $\eta_{1i}$ or even any other non-zero vector. To fix this we consider 

\begin{equation}
\mid (z;\xi)\mid = \int_{\parallel v \parallel_1=1} \left (\sum_s  \epsilon_s \parallel \eta_s \parallel_h^{p/(2s-1)} \right )^{1/p}\mid <\eta_{1}.v> \mid^2
\end{equation}

\noindent
where the integration only affects the last factor making average over all vectors in $v \in V$. This will remove the former difficulty. 
The curvature is the same as for the metric in \cite{D1} but only an extra contribution from the last factor, 

\begin{equation}
\gamma_k(z,\eta)=\frac{i}{2\pi}\left ( w_{r,p}(\eta)+\sum_{lm\alpha}b_{lm\alpha}\left (\int_{\parallel v \parallel_1=1} v_{\alpha}\right )dz_l \wedge d\bar{z}_m +\sum_{s} \frac{1}{s}\frac{\mid \eta_s\mid^{2p/s}}{\sum_t \mid \eta_t \mid^{2p/t}} \sum c_{ij\alpha \beta}\frac{ \eta_{s\alpha}\bar{\eta}_{s\beta}}{\mid \eta_s \mid^2} dz_i \wedge d\bar{z}_j \right )
\end{equation}

\noindent
Namely, if $\pi_r : \mathbb{C}^{kr} \setminus {0} \longrightarrow \mathbb{P}(1^r, 2^r, . . . , k^r)$ is the canonical projection of the weighted projective space $\mathbb{P}(1^r, 2^r, ... , k^r)$ and

\begin{equation} 
\phi_{r,p}(z) := \frac{1}{p}\left (\log (\sum_{k}^{s=1}|z_s|^{\frac{2p}{s}} \right )
\end{equation}

\noindent
for some $p > 0$ then $w_{r,p}$ is the degenerate Kahler form on $\mathbb{P}(1^r, 2^r, . . . , k^r)$ with 

\begin{equation} 
\pi_r^*w_{r,p} = dd^c\phi_{r,p}.
\end{equation} 

\noindent 
We have $b_{l,m,\alpha} \in \mathbb{C}$. The contribution of the factor $\mid \eta_{11} \mid$ can be understood as the curvature of the sub-bundle of $V$ which is orthogonal complement to the remainder. Thus 

\begin{equation}
b_{lm\alpha}=c_{lm\alpha \alpha}
\end{equation} 

\noindent
where $c_{lm11}$ is read from the coefficients of the curvature tensor of $(V,w^{FS})$ the Fubini-Study metric on $V$ (the second summand in (2.12)). Then we need to look at the integral 

\begin{equation}
\int_{X_k,q}\Theta^{n+k(r-1)}=\frac{(n+k(r-1))!}{n!(k(r-1))!}\int_X \int_{P(1^r,...,k^r)}w_{a,r,p}^{k(r-1)}(\eta)1_{\gamma_k,q}(z,\eta) \gamma_k(z,\eta)^n
\end{equation}

\noindent
In the course of evaluating with the Morse inequalities the curvature form is replaced by the trace of the above tensor in raising to the power $n=\dim X$, then if we use polar coordinates

\begin{equation}
x_s=\parallel \eta_s \parallel^{2p/s}, \qquad u_s=\eta_s/\parallel \eta_s \parallel
\end{equation}

\noindent
Then the value of the curvature when integrating over the sphere yields the following

\begin{equation}
\gamma_k=\frac{i}{2\pi}\big(\sum_{lm}b_{lm\alpha} dz_l \wedge d\bar{z}_m+\sum_{s} \frac{1}{s}\sum c_{ij\lambda \lambda} u_{s\alpha}\bar{u}_{s \beta}dz_i \wedge d\bar{z}_j\big)
\end{equation}

\noindent
Because the first term is a finite sum with respect to $1 \leq  \alpha \leq r$, since $b_{lm\alpha}$ are labeled by $\alpha$, the estimates for this new form would be essentially the same as those in \cite{D1}. Therefore one expects

\begin{equation}
\int_{X_k,q}\Theta^{n+k(r-1)}=\frac{(\log k)^n}{n!(k!)^r}\left ( \int_X 1_{\gamma,q}\gamma^n +O((\log k)^{-1}) \right)
\end{equation}

\noindent
where $\gamma$ is the curvature form of $\text{det}(V^*/G_k)$ with respect to the Chern connection of the determinant of the invariant metric and $1_{\gamma,q}$ is the characteristic
function of the set of those $(z, \eta)$, at which $\gamma$ is of signature $(n-q, q)$. We have proved the following.

\begin{proposition}
The analogue of Theorem 1.1 holds if the bundle $E_{k,m}^{GG}$ is replaced by $E_{k,m}$.
\end{proposition}

\begin{proof}
The proof follows from the proof in \cite{D1} and \cite{D2}, and the formulas (2.13-2.19) above.
\end{proof}

\begin{remark}
If $P=P(f,f',...,f^{(k)})$ and $Q=Q(f,f',...,f^{(k)})$ are two local sections of the Green-Griffiths bundle, then the first invariant operator is $f \mapsto f_j'$. Define a bracket operation as follows 

\begin{equation}
[P,Q]=\big(d \log \frac{P^{1/deg(p)}}{Q^{1/deg(Q)}} \big) \times PQ=\frac{1}{deg(P)}QdP-\frac{1}{deg(Q)}PdQ
\end{equation}

\noindent
This is compatible with Merker’s baracket $[P, Q]=\deg(Q)QdP- \deg(P)P dQ$ cf. \cite{M2}. If $(V,h)$ is a Hermitian vector bundle, the equations in (2.20) define inductively $G_k$-equivariant maps 

\begin{equation}
Q_k:J_kV \to S^{k-2} V \otimes \bigwedge^2 V, \qquad Q_k(f)=[f',Q_{k-1}(f)]
\end{equation}

\noindent
The sections produced by $Q_k(f)$ generate the fiber rings of Demailly-Semple bundle. In fact taking charts on the projective fibers one can check that locally the ring that these sections generate are equal to that of $J_k/G_k$, cf. \cite{M2}.
\end{remark}

\vspace{0.3cm}

\section{Existence of global dual differential operators} 

\vspace{0.3cm}

In \cite{M1} J. Merker proves the Green-Griffiths-Lang conjecture for a generic hypersurface in $\mathbb{P}^{n+1}$. He proves for $X \subset \mathbb{P}^{n+1}(\mathbb{C})$ of degree $d$ as a generic member in the universal family

\begin{equation}
\mathfrak{X} \subset \mathbb{P}^{n+1} \times \mathbb{P}^{(\frac{n+1+d)!}{(n+1)!d!}-1}
\end{equation}

\noindent
parametrizing all such hypersurfaces, the GG-conjecture holds. In the proof by Merker for hypersurface case, the Theorem is established outside an algebraic subset $\Sigma \subset J_{\text{vert}}^n(\mathfrak{X})$ defined by vanishing of certain Wronskians, by using a result of Y. T. Siu in \cite{S}. In order to give a similar proof of GG conjecture for general $X$ one may use the following generalization.

\vspace{0.3cm}

\noindent
\textbf{Question:}
If $X \subset \mathbb{P}^{n+1}$ be a generic member of a family $\mathfrak{X}$ of projective varieties, then there are constants $c_n$ and $c_n'$ such that

\begin{equation}
T_{J_{\text{vert}}^n(\mathfrak{X})} \otimes \mathcal{O}_{\mathfrak{X}_k}(c_n) \otimes \pi_{0k}^* L^{c_n'}
\end{equation}

\noindent
is generated at every point by its global sections, where $L$ is an ample line bundle on $\mathfrak{X}$. By the analogy between microlocal differential operators and formal polynomials on the symmetric tensor algebra it suffices to show 

\begin{equation}
H^0(X_k, Sym^{\leq m'}\tilde{V}_k \otimes \mathcal{O}_{X_k}(m) \otimes \pi_{0k}^*B) \ne 0, \qquad m'>>m >>k
\end{equation}

\noindent
where $\tilde{V}_k$ is the in-homogenized $V_k$ as acting as differential operators in first order. We also wish to work over the Demailly-Semple bundle of invariant jets. To this end by a similar procedure as the former case one may check the holomorphic Morse estimates applied to the following metric on the symmetric powers.

\begin{equation}
\vert (z,\xi) \vert= \Big( \sum_{s=1}^k \epsilon_s \big( \sum_{u_i \in S^sV^*} \vert W_{u_1,...,u_s}^s \vert^2 + \sum_{iju_{\alpha}u_{\beta}}C_{iju_{\alpha}u_{\beta}} z_i\bar{z}_j u_{\alpha}\bar{u}_{\beta}\big)^{p/s(s+1)} \Big)^{1/p}
\end{equation} 

\noindent
where $W_{u_1,...,u_s}^s$ is the Wronskian 

\begin{equation}
W_{u_1,...,u_s}^s=W(u_1 \circ f,...,u_s \circ f)
\end{equation}

\noindent
and we regard the summand front the $\epsilon_s$ as a metric on $S^sV^*$. We need to find estimates for the coefficients $C_{iju_{\alpha}u_{\beta}}$. Moreover the frame $\langle u_i \rangle$ is chosen of monomials to be holomorphic and orthonormal at $0$ dual to the frame $\langle e^{\alpha}=\sqrt{l!/\alpha!}e_1^{\alpha_1}...e_r^{\alpha_r}\rangle$. The scaling of the basis in $S^lV^*$ is to make the frame to be orthonormal and are calculated as follows;

\begin{equation}
\langle e^{\alpha}, e^{\beta} \rangle=\langle \sqrt{l!/\alpha!}e_1^{\alpha_1}...e_r^{\alpha_r},\sqrt{l!/\beta!}e_1^{\beta_1}...e_r^{\beta_r} \rangle = \sqrt{1/\alpha! \beta!}\langle \prod_{i=1}^le_{\eta(i)}, \sum_{\sigma \in S_l} \prod_{i=1}^le_{\eta \circ \sigma(i)} \rangle 
\end{equation}

\noindent
via the embedding $S^lV^* \hookrightarrow V^{\otimes l}$ and the map $\eta:\{1....l\} \to \{1....r\}$ taking the value $i$ the $\alpha_i$ times. 

Toward a Morse estimate with the coefficients of the curvature tensor we proceed as follows. Because the frame $\langle e_{\lambda} \rangle$ of $V$ were chosen to be orthonormal at the given point $x \in X$, substituting 

\begin{equation}
\langle e_{\lambda} , e_{\mu} \rangle=\delta_{\lambda \mu}+ \sum_{ij\lambda \mu}c_{ij\lambda \mu} z_i\bar{z}_j+...
\end{equation} 

\noindent
It follows that 

\begin{equation}
\langle e^{\alpha}, e^{\beta} \rangle= \sqrt{1/\alpha! \beta!} \left (\delta_{\alpha \beta} + \sum_{\eta \circ \sigma(i)=\eta(i)} c_{ij\alpha_{\eta(i)}\beta_{\eta(i)}}z_i\bar{z}_j+... \right )
\end{equation}

\noindent
The strategy is to find the scalars $C_{iju_{\alpha}u_{\beta}}$ in terms of curvature of the metric on $V$, in order to examine an estimation of the volume 
\begin{equation}
\int_{X_k} \Theta^{n+k(r-1)} = \frac{(n+k(r-1))!}{n!(k(r-1))!}\int_X \int_{\mathbb{P}(1^{[r]},...,k^{[r]})} \Theta_{\text{vert}}^{k(r-1)}\Theta_{hor}^n
\end{equation}

\noindent
to be positive. However the calculations with $\Theta$ involves more complicated estimates. We pose the question of existence of a positive lower bound for the global sections of the bundle in (3.3) as a step toward the Green-Griffiths conjecture. 

\begin{remark}
In \cite{DR} the existence of dual sections to $H^0(X,E_{k,m}^{GG}V^*)$ for $m \gg k \gg 0$ has been proved using Morse inequalities, involving higher cohomologies of $X_k$ with similar coefficients, cf. loc. cit.. 
\end{remark}


\begin{thebibliography}{99}


\bibitem{S} Y. T. Siu, Hyperbolicity in complex geometry, The legacy of Niels Henrick Abel, Springer, Berlin, 2004, 543-566

\bibitem{D1}  J. P. Demailly, Hyperbolic algebraic varieties and holomorphic differential equations, Acta Math. Vietnam, 37(4) 441-512, 2012

\bibitem{D2}  J. P. Demailly, Holomorphic Morse inequalities and Green-Griffiths-Lang conjecture, preprint

\bibitem{DR} Jean-Pierre Demailly, Mohammad Reza Rahmati, Morse cohomology estimates for jet differential operators, Bollettino dell 'Unione Matematica Italiana, https://doi.org/10.1007/s40574-018-0172-2

\bibitem{GG} M. Green, P. Griffiths, Two applications of algebraic geometry to entire holomorphic mappings, The
Chern Symposium 1979. (Proc. Intern. Sympos., Berkeley, California, 1979) 41-74, Springer, New York, 1980.

\bibitem{M1} J. Merker, low pole order frameson vertical jets of the universal hypersurface, Ann. Institut Fourier (Grenoble), 59, 2009, 861-933 

\bibitem{M2} J. Merker, Application of computational invariant theory to Kobayashi hyperbolicity and to Green–Griffiths algebraic degeneracy, Journal of Symbolic Computation 45 (2010) 986–1074 

\bibitem{P} M. Paun, Vector fields on the total space of hypersurfaces in the projective space and hyperbolicity, Math. Ann. (2008) 340:875-892

\bibitem{RT} M. Rossi, L. Terracini, Weighted projective spaces from toric point of view and computational aspects, preprint

\bibitem{BK} G. Berczi, F. Kirwan, A geometric construction for invariant jet differentials, preprint    

\bibitem{K} I Kaplansky, An introduction to differential algebra, Pub. de L'institut de Math. de Univ. Nancao, Hermann Paris 1957

\end{thebibliography}
\end{document}